\documentclass[12pt,twoside]{amsart}
\usepackage{amsfonts}
\usepackage{amsmath}
\usepackage{amssymb}
\usepackage{graphicx}
\usepackage{mathrsfs}
\usepackage{multicol}
\usepackage{color}
\usepackage{pst-eucl,pstricks-add}
\usepackage{pst-3dplot}
\usepackage{pst-solides3d}

\usepackage{lscape}

\newtheorem{theo}{Theorem}
\newtheorem{cor}{Corollary}
\newtheorem{lem}{Lemma}
\newtheorem{prop}{Proposition}
\theoremstyle{definition}
\newtheorem{defn}{Definition}
\theoremstyle{remark}
\newtheorem{rem}{\bf Remark\/}

\newtheorem{rems}[rem]{\bf Remarks\/}

\numberwithin{equation}{section}

\definecolor{bronze}{rgb}{0.4, 0.7, 0.1}

\def\1{{\mathchoice {\rm 1\mskip-4mu l} {\rm 1\mskip-4mu l}{\rm 1\mskip-4.5mu l} {\rm 1\mskip-5mu l}}}
\newcommand{\ds}{\displaystyle}

\usepackage[width=16.00cm, height=21.00cm, left=2.00cm, right=2.00cm, top=4cm, bottom=2cm]{geometry}

\title[Modified spaces and Segal-Bargmann transforms]{Asymptotic results on modified Bergman-Dirichlet spaces and examples of Segal-Bargmann transforms}

\author[S. Snoun and N. Ghiloufi]{Safa Snoun and Noureddine Ghiloufi}
\email{snoun.safa@fsg.rnu.tn, noureddine.ghiloufi@fsg.rnu.tn}
\address{University of Gabes\\ Faculty of Sciences of Gabes\\ LR17ES11 Mathematics and Applications laboratery\\ 6072, Gabes, Tunisia.}

\subjclass[2010]{30H20, 30C15}
\keywords{Bergman-Dirichlet spaces, Bargmann-Dirichlet spaces, Reproducing Kernels,  Segal-Bargmann transforms}
\begin{document}

\begin{abstract}
    In this paper we start by introducing the  modified Bergman-Dirichlet space $\mathcal D_m^2(\mathbb D_R,\mu^R_{\alpha,\beta})$ and then we study its asymptotic behavior when the parameter $\alpha$ goes to infinity and to $(-1)$ to obtain respectively   the modified Bargmann-Dirichlet and the modified Hardy-Dirichlet spaces with their reproducing kernels. Finally, we give some examples of Segal-Bargmann transforms of those spaces.
\end{abstract}
\maketitle
\section{Introduction}
For every $R>0$ we set $\mathbb D_R=\mathbb D(0,R)$ the disk of $\mathbb C$ with center $0$ and radius $R$ and $\mathbb D_R^*=\mathbb D(0,R)\smallsetminus\{0\}$. For every $\alpha,\beta>-1$  we consider the probability measure $\mu^R_{\alpha,\beta}$ defined on $\mathbb D_R$ by  $$d\mu^R_{\alpha,\beta}(z):=\frac{|z|^{2\beta}\left(R^2-|z|^2\right)^\alpha }{R^{2(\alpha+\beta+1)}\mathscr B(\alpha+1,\beta+1)}dA(z)$$ where $\mathscr B$ is the beta function  and $$dA(z)=\frac{1}{\pi}dx dy=\frac{1}{\pi}rdrd\theta,\quad z=x+iy=re^{i\theta}.$$
When $R=1$ we use $\mathbb D$ and $d\mu_{\alpha,\beta}$ instead of $\mathbb D_R$ and $d\mu^R_{\alpha,\beta}$.\\

In \cite{Gh-Za} the authors introduced $\mathcal A^s(\mathbb D_R,\mu^R_{\alpha,\beta})$, the space of holomorphic functions on the punctured disk $\mathbb D^*_R$  that are $s-$integrable with respect to $\mu^R_{\alpha,\beta}$ and they proved that $\left(\mathcal A^2(\mathbb D_R,\mu^R_{\alpha,\beta}),\| .\|_{\mu^R_{\alpha,\beta}}\right)$ is a Hilbert space where the norm $\| .\|_{\mu^R_{\alpha,\beta}}$ is associated to the inherited inner product $\langle .,.\rangle_{\mu^R_{\alpha,\beta}}$ from $L^2(\mathbb D_R,\mu^R_{\alpha,\beta})$:
$$\langle f,g\rangle_{\mu^R_{\alpha,\beta}}:=\int_{\mathbb D_R}f(z)\overline{g(z)}d\mu^R_{\alpha,\beta}(z),\quad \forall\; f,g\in L^2(\mathbb D_R,\mu^R_{\alpha,\beta}).$$
For more details and further properties of these spaces and their applications in complex and functional analysis, one can see \cite{AG-Sn, Gh-Sn, Gh-Za}.\\

Throughout the paper, $\mathbb N$ will be the set of non-negative integers, i.e. $\mathbb N=\{0,1,\dots\}$ and  $\beta_0\in]-1,0]$ will be a (fixed) real number and we take $\beta=\beta_0+p$ with $p\in\mathbb N$. Many results will be expressed in terms of the hypergeometric functions:
    $$_kF_j\left(\left.
\begin{array}{c}
a_1,\dots,a_k\\
b_1,\dots,b_j
\end{array}\right|\xi\right)=\sum_{n=0}^{+\infty}\frac{(a_1)_n\dots(a_k)_n}{(b_1)_n\dots(b_j)_n}\frac{\xi^n}{n!}$$
where $(a)_n=a(a+1)\dots(a+n-1)$ is the Pochhammer symbol. (See \cite{Ma-Ob-So} for more details about hypergeometric functions).
\\
The paper contains two parts, in the first one we study the modified Bergman, Bargmann and Hardy-Dirichlet spaces while in the second one we give some examples of Segal-Bargmann transforms between subspaces of the mentioned spaces. Precisely, for $m\in\mathbb N$, we consider  the modified Bergman-Dirichlet space $\mathcal D_m^2(\mathbb D_R,\mu^R_{\alpha,\beta})$ formed by all holomorphic functions $f$ on $\mathbb D^*_R$ such that $f^{(m)}\in\mathcal A^2(\mathbb D_R,\mu^R_{\alpha,\beta})$ and we determine explicitly its reproducing kernel in terms of the hypergeometric function $_4F_3$. Namely we have
\begin{theo}\label{th1}
    For every $m\in\mathbb N,\ \alpha>-1$ and $\beta=\beta_0+p$, the reproducing kernel of the Hilbert space $\mathcal D^2_m(\mathbb D_R,\mu^R_{\alpha,\beta})$ is given by $\ds \mathbb K_{\alpha,\beta}^{m,R}(z,w)=\mathcal K_{\alpha,\beta}^{m,R}(z\overline{w})$ where
    $$\mathcal K_{\alpha,\beta}^{m,R}(\xi)=\sum_{n=0}^{m-1}\frac{(\alpha+\beta+2)_n}{(\beta+1)_n}\left(\frac{\xi}{R^2}\right)^n +\frac{\xi^m}{(m!)^2}\ _4F_3\left(\left.
\begin{array}{c}
1,1,1,\alpha+\beta+2\\
m+1,m+1, \beta+1
\end{array}\right|\frac{\xi}{R^2}\right)+A_{\alpha,\beta}^{m,R}(\xi)
    $$
    with
    $$A_{\alpha,\beta}^{m,R}(\xi)=\left\{\begin{array}{lcl}
         0&if& p\leq m\\
         \ds R^{2m}\sum_{k=0}^{p-m-1}\frac{\Gamma(\alpha+\beta+1-m-k)}{\Gamma(\beta-m-k)} \frac{\Gamma(\beta+1)}{\Gamma(\alpha+\beta+2)} \left(\frac{k!}{(m+k)!}\right)^2\left(\frac{R^2}{\xi}\right)^{k+1}& if & p>m.
      \end{array}\right.
    $$
    \end{theo}
    We claim that the result is proved in  \cite{Gh-Sn} (see also \cite{AG-Sn}) when $m=0$  and in \cite{EGIMS} when $\beta=0$.\\

    To introduce the Bargmann-Dirichlet space, we remark that if we take $\alpha=\theta R^2$ for some $\theta>0$, then  using the Stirling formula, we obtain that for every $z\in\mathbb C^*$,
    $$\begin{array}{lcl}
         d\mu^R_{\theta R^2,\beta}(z)&=&\ds\frac{\Gamma(\theta R^2+\beta+2)}{R^2\Gamma(\theta R^2+1)\Gamma(\beta+1)} \left(\frac{|z|^2}{R^2}\right)^\beta\left(1-\frac{|z|^2}{R^2}\right)^{\theta R^2} dA(z)\\
         &\underset{R\to+\infty}\sim& \ds \frac{\theta^{\beta+1}}{\Gamma(\beta+1)}|z|^{2\beta}e^{-\theta|z|^2}dA(z)=:d\nu_{\theta,\beta}(z)
      \end{array}
    $$
    Thus it is essential to consider the modified Bargmann-Fock space  $\mathcal B^2(\mathbb C,\nu_{\theta,\beta})$ as the set of holomorphic functions on  $\mathbb C^*$  that are square integrable with respect to the probability measure $d\nu_{\theta,\beta}$. Similarly as in the previous case, we can define the modified Bargmann-Dirichlet space $\mathcal B^2_m(\mathbb C,\nu_{\theta,\beta})$  of order $m$ as the set of holomorphic functions $f$ on $\mathbb C^*$ such that $f^{(m)}\in\mathcal B^2(\mathbb C,\nu_{\theta,\beta})$.  We prove that $\mathcal B^2_m(\mathbb C,\nu_{\theta,\beta})$ is a Hilbert space and we will determine explicitly its reproducing kernel in terms of the hypergeometric function $_3F_3$. Moreover, we prove that this kernel is exactly the limit of $\ds \mathbb K_{\theta R^2,\beta}^{m,R}$ when $R$ goes to infinity.\\
    If we concentrate on the asymptotic behavior of $\mathcal A^s(\mathbb D,\mu_{\alpha,\beta})$ (for $R=1$) when $\alpha\to-1$, we obtain the modified Hardy  space $H_\beta^s(\mathbb D)$. Furthermore, when $s=2$ we construct the modified Hardy-Dirichlet space $H_{m,\beta}^2(\mathbb D)$  with its reproducing kernel in terms of the hypergeometric function $_3F_2$ (see Theorem \ref{th2}).\\

    Now if $\alpha>-1$ and $\beta_0\in]-1,0]$ are chosen then one can consider  $\mathcal E_p^2(\mathbb D)$ and $\mathcal O_p^2(\mathbb D)$  the subspaces of $\mathcal A^2(\mathbb D,\mu_{\alpha,\beta_0+p})$  generated by the sub-bases $(e_{2n}^p)_{n\geq0}$ and  $(e_{2n+1}^p)_{n\geq0}$ respectively where $(e_{n}^p)_{n\geq0}$ is a Hilbert basis of $\mathcal A^2(\mathbb D,\mu_{\alpha,\beta_0+p})$.  In particular, a function $f\in\mathcal E_p^2(\mathbb D)$ is even when $p$ is even and it is odd when $p$ is odd while a function $g\in\mathcal O_p^2(\mathbb D)$ is odd when $p$ is even and it is even when $p$ is odd. It follows that both $\mathcal E_p^2(\mathbb D)$ and $\mathcal O_p^2(\mathbb D)$ are Hilbert subspaces of $\mathcal A^2(\mathbb D,\mu_{\alpha,\beta_0+p})$ and
    $$\mathcal A^2(\mathbb D,\mu_{\alpha,\beta_0+p})=\mathcal E_p^2(\mathbb D)\oplus^\bot\mathcal O_p^2(\mathbb D).$$
    Analogously, for  fixed $\theta>0$ and $\beta_0\in]-1,0]$, one can consider the two corresponding subspaces $\mathscr E_p^2(\mathbb C)$ and $\mathscr O_p^2(\mathbb C)$ of $\mathcal B^2(\mathbb C,\nu_{\theta,\beta_0+p})$.\\
    The last part of this paper is devoted to give diverse examples of Segal-Bargmann transforms between those subspaces of  $\mathcal A^2(\mathbb D,\mu_{\alpha,\beta_0+p})$ and $\mathcal B^2(\mathbb C,\nu_{\theta,\beta_0+p})$.

\section{Modified Bergman-Dirichlet spaces with asymptotics}
In this section, we start by introducing the modified Bergman-Dirichlet spaces  $\mathcal D^2_m(\mathbb D_R,\mu^R_{\alpha,\beta})$ then we study their asymptotic behavior when $\alpha$ goes to infinity in order to obtain the modified Bargmann-Dirichlet spaces $\mathcal B^2_m(\mathbb C,\nu_{\theta,\beta})$. We finish this part by describing the modified Hardy-Dirichlet spaces $H_\beta^s(\mathbb D)$ as the limit case when $\alpha$ goes to $(-1)$. In each case we determine explicitly the reproducing kernel and we give the relationships between them. The first two functional spaces can be viewed as the modification of the generalized weighted Bergman-Dirichlet and Bargmann-Dirichlet spaces introduced and studied in \cite{EGIMS}.

\subsection{Modified Bergman-Dirichlet spaces of order $m$ on $\mathbb D_R$}
    Let $m\in\mathbb N$ and consider the modified Bergman-Dirichlet space $\mathcal D^2_m(\mathbb D_R,\mu^R_{\alpha,\beta})$ formed by holomorphic functions $f$ on $\mathbb D_R^*$ such that $f^{(m)}\in\mathcal A^2(\mathbb D_R,\mu^R_{\alpha,\beta})$. It is not hard to see that if $f\in\mathcal D^2_m(\mathbb D_R,\mu^R_{\alpha,\beta})$ then $0$ is a removable singularity for $f$  or  a pole with order at most $m-p$ when $p> m$. Thus the function $f$ can be written as
    $$f(z)=\sum_{n=\min(0,m-p)}^{+\infty}a_nz^n.$$
    To define a suitable norm on $\mathcal D^2_m(\mathbb D_R,\mu^R_{\alpha,\beta})$, we decompose $f$  as $f(z)=f_1(z)+f_2(z)$
    where
    \begin{equation}\label{eq2.1}
       \ds f_1(z)=\sum_{n=0}^{m-1}a_nz^n\quad\text{and}\quad
    f_2(z)=\left\{\begin{array}{lcl}
                    \ds\sum_{n=m}^{+\infty}a_nz^n&if& p\leq m\\
                    \ds\sum_{n=m-p}^{-1}a_nz^n+\sum_{n=m}^{+\infty}a_nz^n&if& p> m
                  \end{array}\right.
    \end{equation}
    and we let
    $$\|f\|_{m,\mu^R_{\alpha,\beta}}^2:=\|f_1\|_{\mu^R_{\alpha,\beta}}^2+\left\|f_2^{(m)}\right\|_{\mu^R_{\alpha,\beta}}^2.$$
    This norm is associated with the  inner product defined by $$\langle f,g\rangle_{m,\mu^R_{\alpha,\beta}}=\langle f_1,g_1\rangle_{\mu^R_{\alpha,\beta}}+\left\langle f_2^{(m)},g_2^{(m)}\right\rangle_{\mu^R_{\alpha,\beta}}$$ for  every $f,g\in\mathcal D^2_m(\mathbb D_R,\mu^R_{\alpha,\beta})$.\\

    \begin{lem}\label{lem1}
        If we set $\varphi_n(z)=z^n$,  Then for every $n,k\geq \min(0,m-p)$ we have  $$\langle \varphi_n,\varphi_k\rangle_{m,\mu^R_{\alpha,\beta}}= \left(\varepsilon_{n,\alpha,\beta}^{m,R}\right)^{-2}\delta_{n,k}$$ where $\delta_{n,k}$ is the Kronecker symbol and
        $$\ds \varepsilon_{n,\alpha,\beta}^{m,R}=\left\{\begin{array}{lcl}
             \ds\frac{(n-m)!}{R^{n-m}n!}\sqrt{\frac{(\alpha+\beta+2)_{n-m}}{(\beta+1)_{n-m}}}&if& n\geq m\\
             \ds\frac{1}{R^{(n-m)}}\sqrt{\frac{(\beta+1)_n}{(\alpha+\beta+2)_n}}&if&0\leq n\leq m-1\\
             \ds R^{m-n}\frac{(-n-1)!}{(m-n-1)!} \sqrt{\frac{(\alpha+\beta_0+2)_{n-m+p}}{(\beta_0+1)_{n-m+p}} \frac{(\beta_0+1)_p}{(\alpha+\beta_0+2)_p}}&if& m-p\leq n\leq -1
          \end{array}\right.
        $$
        The last one is valid only if $ p> m$.\\
        In particular, if we set $e_{n,\alpha,\beta}^{m,R}(z):=\varepsilon_{n,\alpha,\beta}^{m,R}z^n$, then $(e_{n,\alpha,\beta}^{m,R})_{n\geq \min(m-p,0)}$ is an orthonormal sequence in $\mathcal D^2_m(\mathbb D_R,\mu^R_{\alpha,\beta})$.
    \end{lem}
    \begin{proof}
        It is easy to see that $\langle \varphi_n,\varphi_k\rangle_{m,\mu^R_{\alpha,\beta}}=0$ if $n\neq k$. Thus it suffices to compute $\|\varphi_n\|_{m,\mu^R_{\alpha,\beta}}^2$ for every $n\geq \min(0,m-p)$. To  this aim we distinguish three cases:
        \begin{enumerate}
          \item \textbf{First statement: $0\leq n\leq m-1$.} In this case we have  $\|\varphi_n\|_{m,\mu^R_{\alpha,\beta}}^2=\|\varphi_n\|_{\mu^R_{\alpha,\beta}}^2$. So using the polar coordinates we obtain
              $$\begin{array}{lcl}
                   \ds\|\varphi_n\|_{m,\mu^R_{\alpha,\beta}}^2&=&\ds \frac{1}{R^2\mathscr B(\alpha+1,\beta+1)}\int_{\mathbb D_R}|z|^{2n}\left(\frac{|z|^2}{R^2}\right)^\beta\left(1-\frac{|z|^2}{R^2}\right)^\alpha dA(z)\\
                   &=&\ds \frac{R^{2n}}{\mathscr B(\alpha+1,\beta+1)}\int_0^R\left(\frac{r^2}{R^2}\right)^{n+\beta}\left(1-\frac{r^2}{R^2}\right)^\alpha \frac{2r}{R^2}dr\\
                   &=&\ds \frac{R^{2n}}{\mathscr B(\alpha+1,\beta+1)}\int_0^1t^{n+\beta}(1-t)^\alpha dt\\
                   &=&\ds \frac{R^{2n}\mathscr B(\alpha+1,n+\beta+1)}{\mathscr B(\alpha+1,\beta+1)}=\frac{R^{2n}(\beta+1)_n}{(\alpha+\beta+2)_n}
                \end{array}$$
          \item \textbf{Second statement: $n\geq m$.} In this case  we have $\|\varphi_n\|_{m,\mu^R_{\alpha,\beta}}^2=\|\varphi_n^{(m)}\|_{\mu^R_{\alpha,\beta}}^2$. Since
          $$\varphi_n^{(m)}(z)=\frac{n!}{(n-m)!}z^{n-m}$$
          then using the same technique as in the first statement, we obtain
          $$\begin{array}{lcl}
                   \ds\|\varphi_n\|_{m,\mu^R_{\alpha,\beta}}^2&=&\ds \frac{R^{2(n-m)}}{\mathscr B(\alpha+1,\beta+1)}\left(\frac{n!}{(n-m)!}\right)^2\int_0^1t^{n-m+\beta}(1-t)^\alpha dt\\
                   &=&\ds \frac{R^{2(n-m)}\mathscr B(\alpha+1,n-m+\beta+1)}{\mathscr B(\alpha+1,\beta+1)}\left(\frac{n!}{(n-m)!}\right)^2\\
                   &=&\ds\frac{R^{2(n-m)}(\beta+1)_{n-m}}{(\alpha+\beta+2)_{n-m}}\left(\frac{n!}{(n-m)!}\right)^2
                \end{array}$$
          \item \textbf{Third statement: $m-p\leq n\leq -1$} (possible only in case $p>m$). Again as in the second statement, we have
          $$\|\varphi_n\|_{m,\mu^R_{\alpha,\beta}}^2=\|\varphi_n^{(m)}\|_{\mu^R_{\alpha,\beta}}^2\text{ and } \varphi_n^{(m)}(z)=(-1)^m\frac{(m-n-1)!}{(-n-1)!}z^{n-m}$$
          Hence we conclude that
          $$\begin{array}{lcl}
                   \ds\|\varphi_n\|_{m,\mu^R_{\alpha,\beta}}^2&=&\ds \frac{R^{2(n-m)}\mathscr B(\alpha+1,n-m+\beta+1)}{\mathscr B(\alpha+1,\beta+1)}\left(\frac{(m-n-1)!}{(-n-1)!}\right)^2\\
                   &=&\ds R^{2(n-m)}\frac{(\beta_0+1)_{n-m+p}}{(\alpha+\beta_0+2)_{n-m+p}} \frac{(\alpha+\beta_0+2)_p}{(\beta_0+1)_p} \left(\frac{(m-n-1)!}{(-n-1)!}\right)^2
                \end{array}$$
        \end{enumerate}
    \end{proof}
    \begin{lem}
    For every $m\in\mathbb N$, we have $\mathcal D^2_m(\mathbb D_R,\mu^R_{\alpha,\beta})\subset \mathcal A^2(\mathbb D_R,\mu^R_{\alpha,\beta})$ and the canonical injection is continuous.\\
    In particular, the space $\mathcal D^2_m(\mathbb D_R,\mu^R_{\alpha,\beta})$ is a Hilbert space and the evaluation form
    $f\longmapsto f(z)$ is continuous on $\mathcal D^2_m(\mathbb D_R,\mu^R_{\alpha,\beta})$ for every $z\in\mathbb D^*_R$.
    \end{lem}
    \begin{proof}
    Let $m\in\mathbb N$. If $m=0$ there is nothing to prove, hence we can assume that $m\geq1$. Let so  $f\in \mathcal D^2_m(\mathbb D_R,\mu^R_{\alpha,\beta})$ with $$f(z)=\sum_{n=\min(0,m-p)}^{+\infty}a_nz^n=\sum_{n=\min(0,m-p)}^{+\infty}a_n\varphi_n(z).$$
    If we set $N_f=0$ when $p\leq m$ and
    $$N_f=\sum_{n=m-p}^{-1}|a_n|^2R^{2(n-m)}\frac{(\beta_0+1)_{n-m+p}}{(\alpha+\beta_0+2)_{n-m+p}} \frac{(\alpha+\beta_0+2)_p}{(\beta_0+1)_p} \left(\frac{(m-n-1)!}{(-n-1)!}\right)^2$$ when $ p>m$, then using Lemma \ref{lem1}, we obtain
    \begin{equation}\label{eq2.2}
    \begin{array}{lcl}
                   \ds\|f\|_{m,\mu^R_{\alpha,\beta}}^2&=&\ds \sum_{n=\min(0,m-p)}^{+\infty}|a_n|^2\|\varphi_n\|_{m,\mu^R_{\alpha,\beta}}^2\\
                   &=&\ds N_f+\sum_{n=0}^{m-1}|a_n|^2\frac{R^{2n}(\beta+1)_n}{(\alpha+\beta+2)_n}\ds+\sum_{n=m}^{+\infty}|a_n|^2 \frac{R^{2(n-m)}(\beta+1)_{n-m}}{(\alpha+\beta+2)_{n-m}}\left(\frac{n!}{(n-m)!}\right)^2
                \end{array}
    \end{equation}
    It is simple to see that when $p>m$, one has
    \begin{equation}\label{eq2.3}
    \begin{array}{l}
         \ds\sum_{n=m-p}^{-1}|a_n|^2\|\varphi_n\|_{\mu^R_{\alpha,\beta}}^2=\ds\sum_{n=m-p}^{-1}|a_n|^2 \frac{R^{2n}(\beta+1)_n}{(\alpha+\beta+2)_n}\\
         =\ds R^{2m}\sum_{n=m-p}^{-1}|a_n|^2R^{2(n-m)}\frac{\Gamma(\beta+1+n-m)}{\Gamma(\alpha+\beta+2+n-m)} \frac{\Gamma(\alpha+\beta+2)}{\Gamma(\beta+1)}\prod_{k=0}^{m-1} \frac{\beta+n-k}{\alpha+\beta+1+n-k}\\
         \leq\ds R^{2m}\sum_{n=m-p}^{-1}|a_n|^2R^{2(n-m)}\frac{\Gamma(\beta+1+n-m)}{\Gamma(\alpha+\beta+2+n-m)} \frac{\Gamma(\alpha+\beta+2)}{\Gamma(\beta+1)} \left(\frac{(m-n-1)!}{(-n-1)!}\right)^2=R^{2m}N_f
      \end{array}
    \end{equation}

      Again for the third term of the right-hand side of \eqref{eq2.2},
      \begin{equation}\label{eq2.4}
      \begin{array}{l}
         \ds\sum_{n=m}^{+\infty}|a_n|^2\|\varphi_n\|_{\mu^R_{\alpha,\beta}}^2=\ds\sum_{n=m}^{+\infty}|a_n|^2 \frac{R^{2n}(\beta+1)_n}{(\alpha+\beta+2)_n}\\
         =\ds R^{2m}\sum_{n=m}^{+\infty}|a_n|^2R^{2(n-m)}\frac{\Gamma(\beta+1+n-m)}{\Gamma(\alpha+\beta+2+n-m)} \frac{\Gamma(\alpha+\beta+2)}{\Gamma(\beta+1)}\prod_{k=0}^{m-1} \frac{\beta+n-k}{\alpha+\beta+1+n-k}\\
         \leq\ds R^{2m}\sum_{n=m}^{+\infty}|a_n|^2R^{2(n-m)}\frac{\Gamma(\beta+1+n-m)}{\Gamma(\alpha+\beta+2+n-m)} \frac{\Gamma(\alpha+\beta+2)}{\Gamma(\beta+1)}\left(\frac{n!}{(n-m)!}\right)^2\\
         \leq\ds R^{2m}\sum_{n=m}^{+\infty}|a_n|^2 \frac{R^{2(n-m)}(\beta+1)_{n-m}}{(\alpha+\beta+2)_{n-m}}\left(\frac{n!}{(n-m)!}\right)^2
      \end{array}
      \end{equation}
      Combining \eqref{eq2.3} and \eqref{eq2.4} with \eqref{eq2.2} we conclude that $f\in \mathcal A^2(\mathbb D_R,\mu^R_{\alpha,\beta})$ and $$\|f\|_{\mu^R_{\alpha,\beta}}\leq \max(1,R^m)\|f\|_{m,\mu^R_{\alpha,\beta}}.$$
    \end{proof}
    Now we can prove Theorem \ref{th1}:
    \begin{proof}
    It is not hard to see that the sequence $(e_{n,\alpha,\beta}^{m,R})_{n\geq \min(m-p,0)}$ is a Hilbert basis for $\mathcal D^2_m(\mathbb D_R,\mu^R_{\alpha,\beta})$. Thus, thanks to \cite{Kr} and using Lemma \ref{lem1}, we obtain that the reproducing kernel is given by
    $$\ds \mathbb K_{\alpha,\beta}^{m,R}(z,w)=\sum_{n=\min(m-p,0)}^{+\infty}e_{n,\alpha,\beta}^{m,R}(z)e_{n,\alpha,\beta}^{m,R}(\overline{w})=\mathcal K_{\alpha,\beta}^{m,R}(z\overline{w})$$
    where
    $$\begin{array}{lcl}
         \ds \mathcal K_{\alpha,\beta}^{m,R}(\xi)&=&\ds A_{\alpha,\beta}^{m,R}(\xi)+ \sum_{n=0}^{m-1}\frac{(\alpha+\beta+2)_n}{(\beta+1)_n}\left(\frac{\xi}{R^2}\right)^n \ds+  \sum_{n=m}^{+\infty}\frac{(\alpha+\beta+2)_{n-m}}{R^{2(n-m)}(\beta+1)_{n-m}}\left(\frac{(n-m)!}{n!}\right)^2\xi^n\\
         &=&\ds A_{\alpha,\beta}^{m,R}(\xi)+ \sum_{n=0}^{m-1}\frac{(\alpha+\beta+2)_n}{(\beta+1)_n}\left(\frac{\xi}{R^2}\right)^n \ds+ \underbrace{\sum_{s=0}^{+\infty}\frac{(\alpha+\beta+2)_s}{R^{2s}(\beta+1)_s}\left(\frac{s!}{(s+m)!}\right)^2\xi^{s+m}}_{ =:F_{\alpha,\beta}^{m,R}(\xi)}
      \end{array}$$
      with $A_{\alpha,\beta}^{m,R}(\xi)=0$ when $ p\leq m$ and
    $$A_{\alpha,\beta}^{m,R}(\xi)=\ds R^{2m}\sum_{k=0}^{p-m-1}\frac{\Gamma(\alpha+\beta+1-m-k)}{\Gamma(\beta-m-k)} \frac{\Gamma(\beta+1)}{\Gamma(\alpha+\beta+2)} \left(\frac{k!}{(m+k)!}\right)^2\left(\frac{R^2}{\xi}\right)^{k+1}$$ when $p>m$.\\
    Using the fact that $s!=(1)_s$ and $(s+m)!=m!(m+1)_s$, one can show that
    $$\ds F_{\alpha,\beta}^{m,R}(\xi)= \ds \frac{\xi^m}{(m!)^2}\sum_{s=0}^{+\infty} \frac{(s!)^3(\alpha+\beta+2)_s}{(\beta+1)_s\left((m+1)_s\right)^2} \frac{1}{s!}\left(\frac{\xi}{R^2}\right)^s=\ds \frac{\xi^m}{(m!)^2}\ _4F_3\left(\left.
\begin{array}{c}
1,1,1,\alpha+\beta+2\\
m+1,m+1, \beta+1
\end{array}\right|\frac{\xi}{R^2}\right)
    $$
    and the proof of Theorem \ref{th1} is accomplished.
    \end{proof}
    \begin{rem}
        Using a result of \cite{Ma-Ob-So} page 47 and the main result of \cite{Gh-Sn}, we obtain for $m=0$:
      \begin{equation}\label{eq2.5}
      \begin{array}{lcl}
           \mathcal K_{\alpha,\beta}^{0,R}(\xi)&=&\ds\frac{(\beta_0+1)_p}{(\alpha+\beta_0+2)_p}\frac{R^{2p}}{\xi^p}\ _2F_1\left(\left.
\begin{array}{cc}
1,&\alpha+\beta_0+2\\
& \beta_0+1
\end{array}\right|\frac{\xi}{R^2}\right).\\
&=&\ds \frac{(\beta_0+1)_p}{(\alpha+\beta_0+2)_p}\frac{R^{2p}}{\xi^p}\frac{1}{\left(1-\frac{\xi}{R^2}\right)^{\alpha+2}}\ _2F_1\left(\left.
\begin{array}{cc}
\beta_0,&-(\alpha+1)\\
& \beta_0+1
\end{array}\right|\frac{\xi}{R^2}\right).\\
&=&\ds\frac{(\beta_0+1)_p}{(\alpha+\beta_0+2)_p}\frac{\beta_0R^{2p}}{\xi^p\left(1-\frac{\xi}{R^2}\right)^{\alpha+2}}\sum_{n=0}^{+\infty} \frac{(-1)^n}{n+\beta_0} {\alpha+1\choose n}\left(\frac{\xi}{R^2}\right)^n.
        \end{array}
      \end{equation}
    \end{rem}
\subsection{Modified Bargmann-Dirichlet spaces of order $m$ on $\mathbb C$}
    Our aim here is to prove an asymptotic result  concerning the limit of the kernel $\mathbb K_{\theta R^2,\beta}^{m,R}$ when $R$ goes to infinity for a $\theta>0$. To this aim we start by recalling the modified Bargmann-Fock space  $\left(\mathcal B^2(\mathbb C,\nu_{\theta,\beta}),\langle .,.\rangle_{\nu_{\theta,\beta}}\right)$ of holomorphic functions on $\mathbb C^*$ that are square integrable with respect to the probability measure $\nu_{\theta,\beta}$ with inner product
    $$\langle f,g\rangle_{\nu_{\theta,\beta}}=\int_{\mathbb C}f(z)\overline{g(z)}d\nu_{\theta,\beta}(z)=\frac{\theta^{\beta+1}}{\Gamma(\beta+1)}\int_{\mathbb C}f(z)\overline{g(z)}|z|^{2\beta}e^{-\theta|z|^2}dA(z)$$

    Similarly to the previous case, we can define the modified Bargmann-Dirichlet space $\mathcal B^2_m(\mathbb C,\nu_{\theta,\beta})$  of order $m\in\mathbb N$ as the set of holomorphic functions $f$ on $\mathbb C^*$ such that $f^{(m)}\in\mathcal B^2(\mathbb C,\nu_{\theta,\beta})$ with the norm
    $$\|f\|_{m,\nu_{\theta,\beta}}^2:=\|f_1\|_{\nu_{\theta,\beta}}^2+\left\|f_2^{(m)}\right\|_{\nu_{\theta,\beta}}^2$$
    for $f(z)=f_1(z)+f_2(z)$ as in \eqref{eq2.1}. This norm is associated with the  inner product defined by $$\langle f,g\rangle_{m,\nu_{\theta,\beta}}=\langle f_1,g_1\rangle_{\nu_{\theta,\beta}}+\left\langle f_2^{(m)},g_2^{(m)}\right\rangle_{\nu_{\theta,\beta}}$$ for  every $f,g\in\mathcal B^2_m(\mathbb C,\nu_{\theta,\beta})$. \\
    Again, if $\varphi_n(z)=z^n$ then $\varphi_n\in\mathcal B^2_m(\mathbb C,\nu_{\theta,\beta})$ for every $n\geq \min(m-p,0)$ and
    $$\|\varphi_n\|_{m,\nu_{\theta,\beta}}^2=
    \left\{\begin{array}{lcl}
            \ds\frac{\Gamma(n+\beta+1)}{\theta^n\Gamma(\beta+1)}&if & 0\leq n\leq m-1\\
            \ds\left(\frac{n!}{(n-m)!}\right)^2\frac{\Gamma(n-m+\beta+1)}{\theta^{n-m}\Gamma(\beta+1)}&if & n\geq m\\
            \ds\left(\frac{(m-n-1)!}{(-n-1)!}\right)^2\frac{\Gamma(n-m+\beta+1)}{\theta^{n-m}\Gamma(\beta+1)}&if & m-p\leq n\leq -1.
           \end{array}\right.
    $$  For every $n\geq \min(m-p,0)$, we set $\mathfrak e_{n,\theta,\beta}^m:=z^n/\|\varphi_n\|_{m,\nu_{\theta,\beta}}$.\\

    \begin{prop}\label{prop1}
    The space $\mathcal B^2_m(\mathbb C,\nu_{\theta,\beta})$ is a Hilbert space and the sequence $(\mathfrak e_{n,\theta,\beta}^m)_{n\geq \min(m-p,0)}$ is a Hilbert basis. Hence the reproducing kernel is given by $\ds \mathbb L_{\theta,\beta}^m(z,w)=\mathcal L_{\theta,\beta}^m(z\overline{w})$ where
    $$\mathcal L_{\theta,\beta}^m(\xi)=\sum_{n=0}^{m-1}\frac{(\theta \xi)^n}{(\beta+1)_n} +\frac{\xi^m}{(m!)^2}\ _3F_3\left(\left.
\begin{array}{ccc}
1,&1,&1\\
m+1,&m+1,& \beta+1
\end{array}\right|\theta\xi\right)+B_{\theta,\beta}^m(\xi)
    $$
    with
    $$B_{\theta,\beta}^m(\xi)=\left\{\begin{array}{lcl}
        \ds \frac{1}{\theta^m}\sum_{k=0}^{p-m-1}\left(\frac{k!}{(m+k)!}\right)^2\frac{\Gamma(\beta+1)}{\Gamma(\beta-m-k)}\frac{1}{(\theta \xi)^{k+1}}&if&p>m\\
        0&if&p\leq m
      \end{array}\right.
    $$
    \end{prop}
    \begin{cor}
        For every $\xi\in\mathbb C^*$ we have
        $$\lim_{R\to+\infty}\mathcal K_{\theta R^2,\beta}^{m,R}(\xi)=\mathcal L_{\theta,\beta}^m(\xi).$$
    \end{cor}
    \begin{proof}
        It is a simple consequence of the Stirling formula applied on Theorem \ref{th1} and Proposition \ref{prop1}.
    \end{proof}
    \begin{rem} Some particular cases:
        \begin{enumerate}
          \item If $m=0=\beta$ then $\mathcal B^2_0(\mathbb C,\nu_{\theta,0})$ is the classical weighted Bargmann-Fock space with reproducing kernel $\ds \mathbb L_{\theta,0}^0(z,w)=e^{\theta z\overline{w}}$.
          \item If $m=0$ and $\beta=\beta_0+p\neq0$ then using Equation \eqref{eq2.5}, we conclude that the reproducing kernel of the modified Bargmann-Fock space $\mathcal B^2_0(\mathbb C,\nu_{\theta,\beta})$ is
              $$\begin{array}{lcl}
                   \ds \mathbb L_{\theta,\beta}^0(z,w)&=&\ds \frac{(\beta_0+1)_p}{(\theta\overline{w}z)^p}\ _1F_1\left(\left.
                \begin{array}{c}
                        1\\
                    \beta_0+1
                \end{array}\right|\theta\overline{w}z\right)\\
                &=&\ds (\beta_0+1)_p\frac{\beta_0e^{\theta\overline{w}z}}{\left(\theta\overline{w}z)\right)^p} \sum_{n=0}^{+\infty}\frac{(-\theta\overline{w}z)^n}{n!(n+\beta_0)}.
                \end{array}
              $$
            This result was proved by \cite{Ha-Yo}, where the reproducing kernel is expressed in terms of the Laguerre polynomials.
          \item If $\beta=0$ and $m\neq0$ then $\mathcal B^2_m(\mathbb C,\nu_{\theta,0})$ is the  Bargmann-Dirichlet space introduced in \cite{EGIMS} with reproducing kernel
              $$\ds \mathbb L_{\theta,0}^m(z,w)=\sum_{n=0}^{m-1}\frac{(\theta z\overline{w})^n}{n!} +\frac{\xi^m}{(m!)^2}\ _2F_2\left(\left.
                \begin{array}{cc}
                    1,&1\\
                    m+1,&m+1
                \end{array}\right|\theta z\overline{w}\right)$$
        \end{enumerate}
    \end{rem}

\subsection{Modified Hardy-Dirichlet spaces}
In this part we begin with the modified Hardy spaces as the limit case of  the modified Bergman spaces $\left(\mathcal A^s(\mathbb D,\mu_{\alpha,\beta}),\| .\|_{s,\mu_{\alpha,\beta}}\right)$ when $\alpha$ goes to $(-1)$. Then we focus on the particular  case $s=2$ (Hilbert space) where we give its reproducing kernel. Similarly to the previous cases, one can  define the modified Hardy-Dirichlet spaces. This statement is so simple, for this reason we use the asymptotic argument to prove our results without details about topological properties of these spaces. \\
We shall start by the classical Hardy spaces $H^s(\mathbb D),\ 0<s<+\infty,$ formed by  holomorphic functions $g$ on $\mathbb D$ such that
$$\|g\|_{H^s(\mathbb D)}:=\sup_{0<r<1}M_s(r,g)<+\infty$$
where $$M_s(r,g):=\left(\frac{1}{2\pi}\int_0^{2\pi}|g(re^{i\theta})|^sd\theta\right)^{\frac{1}{s}}.$$
Due to the fact that the function $r\longmapsto M_s(r,g)$ is non-decreasing on $]0,1[$, we deduce that $$\|g\|_{H^s(\mathbb D)}=\lim_{r\to1}M_s(r,g).$$
Furthermore, the function   $g$ has a radial extension $g^*$ almost every where on the unit circle so that
$$\|g\|_{H^s(\mathbb D)}=\left(\frac{1}{2\pi}\int_0^{2\pi}|g^*(e^{i\theta})|^sd\theta\right)^{\frac{1}{s}}.$$
To introduce the modified Hardy space, for every $-1<\beta<+\infty$ and $0<s<+\infty$, we set
\begin{equation}\label{eq2.6}
\eta_{s,\beta}=\left\{\begin{array}{lcl}
                        \ds\left\lfloor\frac{2(\beta+1)}{s}\right\rfloor & if & \ds\frac{2(\beta+1)}{s}\not\in\mathbb N\\
                        \ds\frac{2(\beta+1)}{s}-1& if & \ds\frac{2(\beta+1)}{s}\in\mathbb N
                      \end{array}\right.
\end{equation}
    where $\lfloor .\rfloor$ is the integer part.
\begin{defn}
 The modified Hardy space $H_\beta^s(\mathbb D)$ is the set of all holomorphic functions $f$ on $\mathbb D^*$ such that $0$ is removable for $f$ or it is  a pole with order $\eta=\eta(f)\leq \eta_{s,\beta}$ (defined in Equation \eqref{eq2.6}) and
 $$\|f\|_{H_\beta^s(\mathbb D)}^s:=\sup_{0<r<1}\frac{r^{s\eta}}{2\pi}\int_0^{2\pi}|f(re^{i\theta})|^sd\theta<+\infty.$$
\end{defn}
\begin{rems}
\begin{enumerate}
  \item The classical Hardy space  $H^s(\mathbb D)$ is exactly $H_0^s(\mathbb D)$.
  \item If $f$ is a given function then $f\in H_\beta^s(\mathbb D)$ if and only if  $\widetilde{f}(z):=z^{\eta(f)}f(z)\in H^s(\mathbb D)$  and $\|f\|_{H_\beta^s(\mathbb D)}=\|\widetilde{f}\|_{H^s(\mathbb D)}$. Thus
      $$\|f\|_{H_\beta^s(\mathbb D)}=\lim_{r\to1^-}M_s(r,\widetilde{f}) =\lim_{r\to1^-}M_s(r,f).$$
      We claim here that the function  $r\longmapsto M_s(r,f)$ can be non-monotone on $]0,1[$ (essentially when $0$ is a pole for $f$ with positive order).
  \item For every $f\in H_\beta^s(\mathbb D)$, the radial extension of $f$ exists and is given by
    $$f^*(e^{i\theta}):=\lim_{r\to1}f(re^{i\theta})=e^{-i\theta\eta(f)}\left(\widetilde{f}\right)^*(e^{i\theta}).$$
    Moreover we have
    $$\|f\|_{H_\beta^s(\mathbb D)}=\left(\frac{1}{2\pi}\int_0^{2\pi}|f^*(e^{i\theta})|^sd\theta\right)^{\frac{1}{s}}.$$

\end{enumerate}
\end{rems}
The construction of the modified Hardy space $H_\beta^s(\mathbb D)$ is inspired from the modified Bergman space $\mathcal A^s(\mathbb D,\mu_{\alpha,\beta})$ for $\alpha$ close to $-1$ as we will see in the following lemma. Indeed we know that if $f\in\mathcal A^s(\mathbb D,\mu_{\alpha,\beta})$ then $0$ can't be an essential singularity for $f$, hence either $0$ is removable for $f$ (so $f$ is holomorphic on $\mathbb D$) or $0$ is a pole for $f$ with order $\eta(f)\leq \eta_{s,\beta}$.
\begin{lem}
If $f\in H_\beta^s(\mathbb D)$ then  $f\in\mathcal A^s(\mathbb D,\mu_{\alpha,\beta})$ for every $\alpha>-1$. Moreover, we have $$\lim_{\alpha\to(-1)^+}\|f\|_{\mathcal A^s(\mathbb D,\mu_{\alpha,\beta})}=\|f\|_{H_\beta^s(\mathbb D)}.$$
\end{lem}
If $\beta=0$ then we cover the result of \cite{Zh}.
\begin{proof}
Let $f\in H_\beta^s(\mathbb D)$. The inclusion  $f\in\mathcal A^s(\mathbb D,\mu_{\alpha,\beta})$ is deduced from the following inequality:
\begin{equation}\label{eq2.7}
    \begin{array}{l}
       \|f\|_{\mathcal A^s(\mathbb D,\mu_{\alpha,\beta})}^s=\ds\int_{\mathbb D}|f(z)|^sd\mu_{\alpha,\beta}(z)\\
       =\ds\frac{1}{\mathscr B(\alpha+1,\beta+1)}\int_{\mathbb D}|f(z)|^s|z|^{2\beta}(1-|z|^2)^\alpha dA(z)\\
       =\ds\frac{2}{\mathscr B(\alpha+1,\beta+1)}\int_0^1r^{2\beta+1}(1-r^2)^\alpha \left(\frac{1}{2\pi}\int_0^{2\pi}|f(re^{i\theta})|^sd\theta\right) dr \\
       =\ds\frac{2}{\mathscr B(\alpha+1,\beta+1)}\int_0^1r^{2\beta+1-s\eta}(1-r^2)^\alpha M_s(r,\widetilde{f})^s dr \\
       \leq \ds \frac{\mathscr B(\alpha+1,\beta-\frac{s\eta}{2}+1)}{\mathscr B(\alpha+1,\beta+1)}\|f\|_{H_\beta^s(\mathbb D)}^s
    \end{array}
\end{equation}
If we tend $\alpha$ to $-1$ in Equation \eqref{eq2.7}, we obtain $$\lim_{\alpha\to(-1)^+}\|f\|_{\mathcal A^s(\mathbb D,\mu_{\alpha,\beta})}\leq\|f\|_{H_\beta^s(\mathbb D)}.$$
To prove the converse inequality, we need the incomplete Beta function $\mathscr B_x(a,b)$ defined by
$$\mathscr B_x(a,b):=\int_0^xt^{a-1}(1-t)^{b-1}dt,\quad a,b>0.$$
Then one has
\begin{equation}\label{eq2.8}
\mathscr B_x(a,b)=\frac{x^a(1-x)^b}{a}\ _2F_1\left(\left.
\begin{array}{cc}
1,&a+b\\
& a+1
\end{array}\right|x\right)
\end{equation}

For $\varepsilon>0$, there exists $r_0\in]0,1[$ such that for every $r\in[r_0,1[$ we have
$\ds M_s(r,\widetilde{f})^s\geq \|f\|_{H_\beta^s(\mathbb D)}^s-\varepsilon$. It follows that
\begin{equation}\label{eq2.9}
    \begin{array}{lcl}
       \ds\|f\|_{\mathcal A^s(\mathbb D,\mu_{\alpha,\beta})}^s
       &\geq&\ds\frac{2}{\mathscr B(\alpha+1,\beta+1)}\int_{r_0}^1r^{2\beta+1-s\eta}(1-r^2)^\alpha M_s(r,\widetilde{f})^s dr \\
       &\geq &\ds \frac{2\left(\|f\|_{H_\beta^s(\mathbb D)}^s-\varepsilon\right)}{\mathscr B(\alpha+1,\beta+1)}\int_{r_0}^1r^{2\beta+1-s\eta}(1-r^2)^\alpha dr\\
       &\geq &\ds \frac{\|f\|_{H_\beta^s(\mathbb D)}^s-\varepsilon}{\mathscr B(\alpha+1,\beta+1)}\int_0^{1-r_0^2}t^\alpha(1-t)^{\beta+1-\frac{s\eta}{2}} dt\\
       &\geq &\ds \frac{\mathscr B_{1-r_0^2}(\alpha+1,\beta+1-\frac{s\eta}{2})}{\mathscr B(\alpha+1,\beta+1)}\left(\|f\|_{H_\beta^s(\mathbb D)}^s-\varepsilon\right)
    \end{array}
\end{equation}
Using Equality \eqref{eq2.8}, we obtain
$$\begin{array}{l}
     \ds\lim_{\alpha\to(-1)^+}\frac{\mathscr B_{1-r_0^2}(\alpha+1,\beta+1-\frac{s\eta}{2})}{\mathscr B(\alpha+1,\beta+1)}\\
     =\ds \lim_{\alpha\to(-1)^+}\frac{(1-r_0^2)^{\alpha+1}r_0^{2\beta+2-s\eta}}{(\alpha+1)\mathscr B(\alpha+1,\beta+1)}\ _2F_1\left(\left.
\begin{array}{cc}
1,&\alpha+\beta+2-\frac{s\eta}{2}\\
& \alpha+2
\end{array}\right|1-r_0^2\right)\\
=\ds r_0^{2\beta+2-s\eta}\frac{1}{\left(1-(1-r_0^2)\right)^{\beta+1-\frac{s\eta}{2}}}=1.
  \end{array}
$$
Thanks to Inequality \eqref{eq2.9}, we obtain $$ \lim_{\alpha\to(-1)^+}\|f\|_{\mathcal A^s(\mathbb D,\mu_{\alpha,\beta})}^s\geq \|f\|_{H_\beta^s(\mathbb D)}^s-\varepsilon.$$
Since $\varepsilon$ is chosen arbitrary, we conclude that $$ \lim_{\alpha\to(-1)^+}\|f\|_{\mathcal A^s(\mathbb D,\mu_{\alpha,\beta})}^s\geq \|f\|_{H_\beta^s(\mathbb D)}^s$$ and the proof is achieved.
\end{proof}
Now we restrict our study to the case $s=2$ and $\beta=\beta_0+p$ as above (that gives $\eta_{2,\beta}=p$). Thus the modified Hardy space $H_\beta^2(\mathbb D)$ is a Hilbert space where  the corresponding norm
$\|f\|_{H_\beta^2(\mathbb D)}=\sqrt{\langle f,f\rangle_{H_\beta^2}}$ is produced by the inner product
$$\langle f,g\rangle_{H_\beta^2}=\frac{1}{2\pi}\int_0^{2\pi}f^*(e^{i\theta})\overline{g^*(e^{i\theta})}d\theta,\quad \forall\; f,g\in H_\beta^2(\mathbb D).$$
Similarly as in the previous cases, we can define the modified Hardy-Dirichlet space $H_{m,\beta}^2(\mathbb D)$  of order $m\in\mathbb N$ as the set of holomorphic functions $f$ on $\mathbb C^*$ such that $f^{(m)}\in H_\beta^2(\mathbb D)$ with the norm
    $$\|f\|_{H_{m,\beta}^2}^2:=\|f_1\|_{H_\beta^2}^2+\left\|f_2^{(m)}\right\|_{H_\beta^2}^2$$
    for $f(z)=f_1(z)+f_2(z)$ as in \eqref{eq2.1}. This norm is associated with the  inner product defined by $$\langle f,g\rangle_{H_{m,\beta}^2}=\langle f_1,g_1\rangle_{H_\beta^2}+\left\langle f_2^{(m)},g_2^{(m)}\right\rangle_{H_\beta^2}$$ for  every $f,g\in H_{m,\beta}^2(\mathbb D) $.

    \begin{theo}\label{th2}
    The reproducing kernel of the Hilbert space $H^2_{m,\beta}(\mathbb D)$ is given by $\ds \mathbb H_{m,\beta}(z,w)=\mathcal H_{m,\beta}(z\overline{w})$ where
    $$\begin{array}{lcl}
         \mathcal H_{m,\beta}(\xi)&=&\ds \lim_{\alpha\to-1}\mathcal K_{\alpha,\beta}^{m,1}(\xi)\\
         &=&\ds\frac{1-\xi^m}{1-\xi} +\frac{\xi^m}{(m!)^2}\ _3F_2\left(\left.
\begin{array}{c}
1,\quad 1,\quad 1\\
m+1,m+1
\end{array}\right|\xi\right)+C_{m,\beta}(\xi)
      \end{array}
    $$
    with
    $$C_{m,\beta}(\xi)=\left\{\begin{array}{lcl}
        \ds \sum_{k=0}^{p-m-1} \left(\frac{k!}{(m+k)!}\right)^2\frac{1}{\xi^{k+1}}&if& p>m\\
        0& if& p\leq m
      \end{array}\right.
    $$
    \end{theo}
    As a particular case, if $m=0$ we conclude that the reproducing kernel of the Hilbert space $H^2_{\beta}(\mathbb D)$ is given by $\mathbb H_{\beta}(z,w)=\mathcal H_{\beta}(z\overline{w})$ where
    $$\mathcal H_{\beta}(\xi)=\frac{1}{\xi^p(1-\xi)}.$$
\section{Segal-Bargmann transforms}
    The Segal-Bargmann transform is an isometry invertible transform between two separable Hilbert spaces. This type of transforms was studied in many papers. To introduce this transform, we restrict our study to the case of two separable complex functional Hilbert spaces $(\mathscr H_X,d\mu_X)$ and $(\mathscr H_Y,d\mu_Y)$  on two sets $X,\ Y$ with corresponding orthonormal basis $(e_n^X)_n$ and $(e_n^Y)_n$ with respect to the inner products
    $$\langle f,g\rangle_X=\int_X f(x)\overline{g}(x)d\mu_X(x)\quad \text{and} \quad \langle h,k\rangle_Y=\int_Y h(y)\overline{k}(y)d\mu_Y(y)$$
    respectively. The Segal-Bargmann (S-B for short) transform $\mathcal T$ that transforms isometrically $(\mathscr H_X,d\mu_X)$ onto $(\mathscr H_Y,d\mu_Y)$ is given by $$\mathcal T \varphi(y)=\int_X T(y,x)\varphi(x)d\mu_X(x)$$
    for every $\varphi\in \mathscr H_X$ where the kernel $T$ is defined by
    $$T(y,x)=\sum_{n=0}^{+\infty}e_n^Y(y)\overline{e_n^X(x)}.$$
    Indeed, for every $m\geq 0$, one has
    $$\mathcal T e_m^X(y)=\sum_{n=0}^{+\infty}e_n^Y(y)\langle e_m^X,e_n^X(x)\rangle_X =e_m^Y(y)$$
    It follows that for every $\varphi\in \mathscr H_X,\ \varphi(x)=\sum_{n=0}^{+\infty}a_ne_n^X(x) $ we have
    $$\mathcal T\varphi(y)=\sum_{m=0}^{+\infty}a_m\mathcal Te_m^X(y)=\sum_{m=0}^{+\infty} a_me_m^Y(y).$$
    In particular,
    $$||\varphi||_{\mathscr H_X}^2=\sum_{m=0}^{+\infty} |a_m|^2=||\mathcal T\varphi||_{\mathscr H_Y}^2.$$
    Thus $\mathcal T$ is an isometry from $(\mathscr H_X,d\mu_X)$ onto $(\mathscr H_Y,d\mu_Y)$. (For more details and further examples of Segal-Bargmann transforms, one can see \cite{Be-Gh}).\\

    As an application of the Segal-Bargmann transforms, we will construct a particular form of invertible isometries from a separable Hilbert space $\mathscr H_X$ onto itself in the case where $\mathscr H_X=\mathcal E\oplus^\bot\mathcal O$ is the orthogonal sum of two separable Hilbert subspaces $\mathcal E$ and $\mathcal O$. In fact, if $\mathcal T_1$ is a Segal-Bargmann transform from $\mathcal E$ onto $\mathcal O$ and $\mathcal T_2$ is a Segal-Bargmann transform from $\mathcal O$ onto $\mathcal E$ then there exists a unique linear transform $\mathbb T$ from $\mathscr H_X$ onto itself that is an invertible isometry such that $\mathbb T_{|\mathcal E}=\mathcal T_1$ and $\mathbb T_{|\mathcal O}=\mathcal T_2$ with $\mathbb T^2=\mathbb T\circ \mathbb T=Id_{\mathscr H_X}$.\\

\subsection{Case of modified Bergman spaces}
In this part we restrict our work to the case $R=1$ and $m=0$, we omit so the indexation by $R$ and $m$ in all notations.\\

    The fundamental aim of this part is to determine explicitly the kernels of Segal-Bargmann transforms of some subspaces of the modified Bergman space $\mathcal A^2(\mathbb D,\mu_{\alpha,\beta})$.\\
    For every  $p,q,n\geq0$, we denote by
    \begin{equation}\label{eq3.1}
        \left\{\begin{array}{lcl}
                  \gamma_n^p&=&\ds\sqrt{\frac{\mathscr B(\alpha+1,\beta_0+1+p)}{\mathscr B(\alpha+1,n+\beta_0+1)}}=\sqrt{\frac{(\beta_0+1)_p}{(\alpha+\beta_0+2)_p}\frac{(\alpha+\beta_0+2)_n}{(\beta_0+1)_n}}\\
                  c_{p,q}&=&\ds\sqrt{\frac{(\beta_0+1)_p\ (\beta_0+1)_q}{(\alpha+\beta_0+2)_p\ (\alpha+\beta_0+2)_q}}.
               \end{array}\right.        
    \end{equation}    
    The following proposition will be useful in the rest.
\begin{prop}
For every $p,q\in\mathbb N$ and $\xi\in\mathbb D$, we have
\begin{equation}\label{eq3.2}
\sum_{n=0}^{+\infty} \gamma_n^p\gamma_n^q\xi^n=c_{p,q}\ _2F_1\left(\left.
\begin{array}{cc}
1,&\alpha+\beta_0+2\\
& \beta_0+1
\end{array}\right|\xi\right)
\end{equation}
\begin{equation}\label{eq3.3}
\sum_{n=0}^{+\infty} \gamma_{2n}^p\gamma_{2n}^q\xi^{2n}=c_{p,q}\ _3F_2\left(\left.
\begin{array}{ccc}
1,&\ds\frac{\alpha+\beta_0+2}{2},&\ds\frac{\alpha+\beta_0+3}{2}\\
& \ds\frac{\beta_0+1}{2},& \ds\frac{\beta_0+2}{2}
\end{array}\right|\xi^2\right).
\end{equation}
\begin{equation}\label{eq3.4}
\ds\sum_{n=0}^{+\infty} \gamma_{2n+1}^p\gamma_{2n+1}^q\xi^{2n}=\frac{\alpha+\beta_0+2}{\beta_0+1} c_{p,q}\ _3F_2\left(\left.
\begin{array}{ccc}
1,&\ds\frac{\alpha+\beta_0+3}{2},&\ds\frac{\alpha+\beta_0+4}{2}\\
& \ds\frac{\beta_0+2}{2},& \ds\frac{\beta_0+3}{2}
\end{array}\right|\xi^2\right).
\end{equation}
\begin{equation}\label{eq3.5}
    \ds\sum_{n=0}^{+\infty} \gamma_{2n}^p\gamma_{2n+1}^q\xi^{2n}=\sqrt{\frac{\alpha+\beta_0+2}{\beta_0+1}} c_{p,q}\ _3F_2\left(\left.
\begin{array}{ccc}
1,&\ds\frac{\alpha+\beta_0+3}{2},&\ds\frac{\alpha+\beta_0+4}{2}\\
& \ds\frac{\beta_0+2}{2},& \ds\frac{\beta_0+3}{2}
\end{array}\right|\xi^2\right).
\end{equation}
\end{prop}

\begin{proof}
For every $\xi\in\mathbb D$ we have
$$
\begin{array}{lcl}
\ds\sum_{n=0}^{+\infty} \gamma_n^p\gamma_n^q\xi^n&=&\ds\sqrt{\frac{(\beta_0+1)_p}{(\alpha+\beta_0+2)_p}}\sqrt{\frac{(\beta_0+1)_q}{(\alpha+\beta_0+2)_q}}\sum_{n=0}^{+\infty} \frac{(\alpha+\beta_0+2)_n}{(\beta_0+1)_n}\xi^n\\
&=&\ds c_{p,q}\sum_{n=0}^{+\infty} (1)_n\frac{(\alpha+\beta_0+2)_n}{(\beta_0+1)_n}\frac{\xi^n}{n!}\\
&=&\ds c_{p,q}\ _2F_1\left(\left.
\begin{array}{cc}
1,&\alpha+\beta_0+2\\
& \beta_0+1
\end{array}\right|\xi\right).
\end{array}
$$
For the second formula, we have
$$\begin{array}{lcl}
\ds\sum_{n=0}^{+\infty} \gamma_{2n}^p\gamma_{2n}^q\xi^{2n}&=&\ds \sqrt{\frac{(\beta_0+1)_p}{(\alpha+\beta_0+2)_p}}\sqrt{\frac{(\beta_0+1)_q}{(\alpha+\beta_0+2)_q}}\sum_{n=0}^{+\infty} \frac{(\alpha+\beta_0+2)_{2n}}{(\beta_0+1)_{2n}}\xi^{2n}
\end{array}
$$
Using the formula
$$(a)_{2n}=2^{2n}\left(\frac{a}{2}\right)_n\left(\frac{a+1}{2}\right)_n$$
we obtain
$$\begin{array}{lcl}
\ds\sum_{n=0}^{+\infty} \gamma_{2n}^p\gamma_{2n}^q\xi^{2n}&=&\ds c_{p,q}\sum_{n=0}^{+\infty} \frac{(1)_n\left(\frac{\alpha+\beta_0+2}{2}\right)_n\left(\frac{\alpha+\beta_0+3}{2}\right)_n}{\left(\frac{\beta_0+1}{2}\right)_n\left(\frac{\beta_0+2}{2}\right)_n}\frac{\xi^{2n}}{n!}\\
&=&\ds c_{p,q}\ _3F_2\left(\left.
\begin{array}{ccc}
1,&\ds\frac{\alpha+\beta_0+2}{2},&\ds\frac{\alpha+\beta_0+3}{2}\\
& \ds\frac{\beta_0+1}{2},& \ds\frac{\beta_0+2}{2}
\end{array}\right|\xi^2\right).
\end{array}
$$
To prove the third equation, we claim that
$$(a)_{2n+1}=\ds \frac{\Gamma(a+2n+1)}{\Gamma(a)}=a\frac{\Gamma(a+1+2n)}{\Gamma(a+1)}
    =\ds a(a+1)_{2n} =a2^{2n}\left(\frac{a+1}{2}\right)_n\left(\frac{a+2}{2}\right)_n.
$$
Hence
$$\begin{array}{lcl}
\ds\sum_{n=0}^{+\infty} \gamma_{2n+1}^p\gamma_{2n+1}^q\xi^{2n}
&=&\ds c_{p,q}\sum_{n=0}^{+\infty} \frac{(\alpha+\beta_0+2)_{2n+1}}{(\beta+1)_{2n+1}}\xi^{2n}\\
&=&\ds c_{p,q}\sum_{n=0}^{+\infty}\frac{\ds(\alpha+\beta_0+2)\left(\frac{\alpha+\beta_0+3}{2}\right)_n\left(\frac{\alpha+\beta_0+4}{2}\right)_n}{\ds(\beta_0+1) \left(\frac{\beta_0+2}{2}\right)_n\left(\frac{\beta_0+3}{2}\right)_n}\xi^{2n}\\
&=&\ds \frac{\alpha+\beta_0+2}{\beta_0+1} c_{p,q}\ _3F_2\left(\left.
\begin{array}{ccc}
1,&\ds\frac{\alpha+\beta_0+3}{2},&\ds\frac{\alpha+\beta_0+4}{2}\\
& \ds\frac{\beta_0+2}{2},& \ds\frac{\beta_0+3}{2}
\end{array}\right|\xi^2\right).
\end{array}
$$
Finally, for the last equation,
$$\begin{array}{lcl}
\ds\sum_{n=0}^{+\infty} \gamma_{2n}^p\gamma_{2n+1}^q\xi^{2n}&=&\ds c_{p,q} \sqrt{\frac{\alpha+\beta_0+2}{\beta_0+1}}\sum_{n=0}^{+\infty}\frac{\ds \left(\frac{\alpha+\beta_0+3}{2}\right)_n\left(\frac{\alpha+\beta_0+4}{2}\right)_n}{ \ds \left(\frac{\beta_0+2}{2}\right)_n\left(\frac{\beta_0+3}{2}\right)_n}\xi^{2n}\\
&=&\ds  c_{p,q}\sqrt{\frac{\alpha+\beta_0+2}{\beta_0+1}}\ _3F_2\left(\left.
\begin{array}{ccc}
1,&\ds\frac{\alpha+\beta_0+3}{2},&\ds\frac{\alpha+\beta_0+4}{2}\\
& \ds\frac{\beta_0+2}{2},& \ds\frac{\beta_0+3}{2}
\end{array}\right|\xi^2\right).
\end{array}
$$
\end{proof}

Now we can cite the first result concerning the Segal-Bargmann transform between modified Bergman spaces.

\begin{theo}
For every $p,q\in\mathbb N$, we consider the integral transform $D_{p,q}$ given by
$$D_{p,q}f(z)=c_{p,q}\int_{\mathbb D}\frac{f(w)}{\overline{w}^pz^q}\ _2F_1\left(\left.
\begin{array}{cc}
1,&\alpha+\beta_0+2\\
& \beta_0+1
\end{array}\right|\overline{w}z\right)d\mu_{\alpha,\beta_0+p}(w).
$$
Then $D_{p,q}$ is a well defined bounded operator on the Hilbert space $\mathcal A^2(\mathbb D,\mu_{\alpha,\beta_0+p})$ and maps it isometrically onto $\mathcal A^2(\mathbb D,\mu_{\alpha,\beta_0+q})$. Its inverse $D_{p,q}^{-1}=D_{q,p}:\mathcal A^2(\mathbb D,\mu_{\alpha,\beta_0+q})\longrightarrow \mathcal A^2(\mathbb D,\mu_{\alpha,\beta_0+p})$ is given by
$$D_{p,q}^{-1}f(w)=c_{p,q}\int_{\mathbb D}\frac{f(z)}{\overline{z}^qw^p}\ _2F_1\left(\left.
\begin{array}{cc}
1,&\alpha+\beta_0+2\\
& \beta_0+1
\end{array}\right|\overline{z}w\right)d\mu_{\alpha,\beta_0+q}(z).
$$
\end{theo}
\begin{proof}
If we set $e_n^p(z)=\gamma_n^p\ z^{n-p}$, then the sequence $(e_n^p)_{n\geq 0}$ is an orthonormal basis of $\mathcal A^2(\mathbb D,\mu_{\alpha,\beta_0+p})$. It follows that the Segal-Bargmann transform $D_{p,q}: \mathcal A^2(\mathbb D,\mu_{\alpha,\beta_0+p})\longrightarrow \mathcal A^2(\mathbb D,\mu_{\alpha,\beta_0+q})$ that transforms  $e_n^p$ to $e_n^q$ can be written as
$$D_{p,q}f(z)=\int_{\mathbb D}f(w)\mathfrak{D}_{p,q}(z,w)d\mu_{\alpha,\beta_0+p}(w)$$
where the kernel $\mathfrak{D}_{p,q}$ is given by
$$\mathfrak{D}_{p,q}(z,w)=\sum_{n=0}^{+\infty}\overline{e_n^p(w)}e_n^q(z)=\frac{1}{\overline{w}^pz^q}\sum_{n=0}^{+\infty}\gamma_n^p\gamma_n^q(\overline{w}z)^n.
$$
Thus, thanks to Formula \eqref{eq3.2},
$$\mathfrak{D}_{p,q}(z,w)=\frac{c_{p,q}}{\overline{w}^pz^q}\ _2F_1\left(\left.
\begin{array}{cc}
1,&\alpha+\beta_0+2\\
& \beta_0+1
\end{array}\right|\overline{w}z\right).
$$ Hence the proof is finished.
\end{proof}
\begin{rem}
By taking $p=q$ in the previous theorem,  we reobtain the reproducing kernel of $\mathcal A^2(\mathbb D,\mu_{\alpha,\beta_0+p})$.
\end{rem}
In what follows, we discuss the Segal-Bargmann transforms between (even or odd) subspaces of modified Bergman spaces. Hence, three cases are possible: even-even, odd-odd and even-odd. Each case will be given in a separate result as follows.
\begin{theo}
For every $p,q\in\mathbb N$, we consider the Segal-Bargmann transform $G_{p,q}$ given by
$$G_{p,q}f(z)=c_{p,q}\int_{\mathbb D}\frac{f(w)}{\overline{w}^pz^q}\ _3F_2\left(\left.
\begin{array}{ccc}
1,&\frac{\alpha+\beta_0+2}{2},&\frac{\alpha+\beta_0+3}{2}\\
& \frac{\beta_0+1}{2},& \frac{\beta_0+2}{2}
\end{array}\right|(\overline{w}z)^2\right)d\mu_{\alpha,\beta_0+p}(w).
$$
Then $G_{p,q}$ maps isometrically the Hilbert space $\mathcal E_p^2(\mathbb D)$ onto $\mathcal E_q^2(\mathbb D)$ and its inverse $G_{p,q}^{-1}:\mathcal E_q^2(\mathbb D)\longrightarrow \mathcal E_p^2(\mathbb D)$ is given by
$$G_{p,q}^{-1}f(w)=c_{p,q}\int_{\mathbb D}\frac{f(z)}{\overline{z}^qw^p}\ _3F_2\left(\left.
\begin{array}{ccc}
1,&\frac{\alpha+\beta_0+2}{2},&\frac{\alpha+\beta_0+3}{2}\\
& \frac{\beta_0+1}{2},&\frac{\beta_0+2}{2}
\end{array}\right|(\overline{z}w)^2\right)d\mu_{\alpha,\beta_0+q}(z).
$$
\end{theo}
\begin{proof}
With the choice of $\mathcal E_p^2(\mathbb D)$, the sequence $(e_{2n}^p)_{n\in\mathbb N}$ is a Hilbert basis of $\mathcal E_p^2(\mathbb D)$. It follows that the kernel of $G_{p,q}$ is given by
$$\begin{array}{lcl}
\ds\mathfrak{G}_{p,q}(z,w)&=&\ds\sum_{n=0}^{+\infty}\overline{e_{2n}^p(w)}e_{2n}^q(z)\\
&=&\ds\frac{1}{\overline{w}^pz^q}\sum_{n=0}^{+\infty}\gamma_{2n}^p\gamma_{2n}^q(\overline{w}z)^{2n}\\
&=&\ds \frac{c_{p,q}}{\overline{w}^pz^q}\ _3F_2\left(\left.
\begin{array}{ccc}
1,&\frac{\alpha+\beta_0+2}{2},&\frac{\alpha+\beta_0+3}{2}\\
& \frac{\beta_0+1}{2},& \frac{\beta_0+2}{2}
\end{array}\right|(\overline{w}z)^2\right)
\end{array}
$$
where we have used Equality \eqref{eq3.3} and so the proof is finished.
\end{proof}
\begin{theo}
For every $p,q\in\mathbb N$, the Segal Bargmann transform $J_{p,q}$ given by
$$J_{p,q}f(z)=\int_{\mathbb D}f(w)\mathfrak{I}_{p,q}(z,w)d\mu_{\alpha,\beta_0+p}(w).
$$
 maps isometrically the Hilbert space $\mathcal O_p^2(\mathbb D)$ onto $\mathcal O_q^2(\mathbb D)$ and its inverse $J_{p,q}^{-1}=J_{q,p}:\mathcal O_q^2(\mathbb D)\longrightarrow \mathcal O_p^2(\mathbb D)$ is given by
$$J_{p,q}^{-1}f(w)=\int_{\mathbb D}f(z)\mathfrak{J}_{q,p}(w,z)d\mu_{\alpha,\beta_0+q}(z).
$$
with
$$\mathfrak{J}_{p,q}(z,w)=\frac{\alpha+\beta_0+2}{\beta_0+1}\frac{c_{p,q}}{\overline{w}^{p-1}z^{q-1}}\ _3F_2\left(\left.
\begin{array}{ccc}
1,&\frac{\alpha+\beta_0+3}{2},&\frac{\alpha+\beta_0+4}{2}\\
& \frac{\beta_0+2}{2},& \frac{\beta_0+3}{2}
\end{array}\right|(\overline{w}z)^2\right).
$$
\end{theo}
\begin{proof}
As in the previous result, the sequence $(e_{2n+1}^p)_{n\in\mathbb N}$ is a Hilbert basis of $\mathcal O_p^2(\mathbb D)$. Hence, if we apply Formula \eqref{eq3.4} we can conclude that the kernel of $J_{p,q}$ is given by
$$\begin{array}{l}
\ds\mathfrak{J}_{p,q}(z,w)=\ds\sum_{n=0}^{+\infty}\overline{e_{2n+1}^p(w)}e_{2n+1}^q(z)\\
=\ds\frac{1}{\overline{w}^{p-1}z^{q-1}}\sum_{n=0}^{+\infty}\gamma_{2n+1}^p\gamma_{2n+1}^q(\overline{w}z)^{2n}\\
=\ds \frac{\alpha+\beta_0+2}{\beta_0+1}\frac{c_{p,q}}{\overline{w}^{p-1}z^{q-1}}\ _3F_2\left(\left.
\begin{array}{ccc}
1,&\frac{\alpha+\beta_0+3}{2},&\frac{\alpha+\beta_0+4}{2}\\
& \frac{\beta_0+2}{2},& \frac{\beta_0+3}{2}
\end{array}\right|(\overline{w}z)^2\right).
\end{array}
$$
\end{proof}

\begin{theo}
For every $p,q\in\mathbb N$, the Segal-Bargmann transform $S_{p,q}$ given by
$$S_{p,q}f(z)=\int_{\mathbb D}f(w)\mathfrak{s}_{p,q}(z,w)d\mu_{\alpha,\beta_0+p}(w)
$$
maps isometrically the Hilbert space $\mathcal E_p^2(\mathbb D)$ onto $\mathcal O_q^2(\mathbb D)$ and its inverse $S_{p,q}^{-1}:\mathcal O_q^2(\mathbb D)\longrightarrow \mathcal E_p^2(\mathbb D)$ is given by
$$S_{p,q}^{-1}f(w)=\int_{\mathbb D}f(z)\mathfrak{s}_{q,p}(w,z)d\mu_{\alpha,\beta_0+q}(z).
$$
with
$$\mathfrak{s}_{p,q}(z,w)=\sqrt{\frac{\alpha+\beta_0+2}{\beta_0+1}}\frac{c_{p,q}}{\overline{w}^{p}z^{q-1}}\ _3F_2\left(\left.
\begin{array}{ccc}
1,&\frac{\alpha+\beta_0+3}{2},&\frac{\alpha+\beta_0+4}{2}\\
& \frac{\beta_0+2}{2},& \frac{\beta_0+3}{2}
\end{array}\right|(\overline{w}z)^2\right).
$$
\end{theo}
\begin{proof}
Since the sequences $(e_{2n}^p)_{n\in\mathbb N}$ and $(e_{2n+1}^q)_{n\in\mathbb N}$ are Hilbert bases of $\mathcal E_p^2(\mathbb D)$ and $\mathcal O_q^2(\mathbb D)$ respectively, then by Formula \eqref{eq3.5} we can conclude that the kernel of $S_{p,q}$ is given by
$$\begin{array}{l}
\ds\mathfrak{s}_{p,q}(z,w)=\ds\sum_{n=0}^{+\infty}\overline{e_{2n}^p(w)}e_{2n+1}^q(z)\\
=\ds\frac{1}{\overline{w}^{p}z^{q-1}}\sum_{n=0}^{+\infty}\gamma_{2n}^p\gamma_{2n+1}^q(\overline{w}z)^{2n}\\
=\ds \sqrt{\frac{\alpha+\beta_0+2}{\beta_0+1}} \frac{c_{p,q}}{\overline{w}^{p}z^{q-1}}\ _3F_2\left(\left.
\begin{array}{ccc}
1,&\frac{\alpha+\beta_0+3}{2},&\frac{\alpha+\beta_0+4}{2}\\
& \frac{\beta_0+2}{2},& \frac{\beta_0+3}{2}
\end{array}\right|(\overline{w}z)^2\right).
\end{array}
$$
\end{proof}

As an application of the last result, and using the fact that $\mathcal A^2(\mathbb D,\mu_{\alpha,\beta_0+p})=\mathcal E_p^2(\mathbb D)\oplus^\bot\mathcal O_p^2(\mathbb D),$ we can deduce the following corollary:
\begin{cor}
For every $p\in\mathbb N$, we consider the integral transform $T$ given by
$$Tf(z):=\int_{\mathbb D}f(w)\mathfrak T_p(z,w)d\mu_{\alpha,\beta_0+p}(w)
$$
 where
 $$\begin{array}{lcl}
      \mathfrak T_p(z,w)&=&\ds\sqrt{\frac{\alpha+\beta_0+2}{\beta_0+1}}\frac{(\beta_0+1)_p}{(\alpha+\beta_0+2)_p}\times\\
      & &\ds\Re e\left\{\frac{1}{\overline{w}^{p}z^{q-1}}\ _3F_2\left(\left.
\begin{array}{ccc}
1,&\ds\frac{\alpha+\beta_0+3}{2},&\ds\frac{\alpha+\beta_0+4}{2}\\
& \ds\frac{\beta_0+2}{2},& \ds\frac{\beta_0+3}{2}
\end{array}\right|(\overline{w}z)^2\right)\right\}.
   \end{array}
$$
 Then $T$ is an isometry from $\mathcal A^2(\mathbb D,\mu_{\alpha,\beta_0+p})$ onto itself that verify $T^2=Id_{\mathcal A^2(\mathbb D,\mu_{\alpha,\beta_0+p})}$ and  $T_{|\mathcal E_p^2(\mathbb D)}=S_{p,p}$ and $T_{|\mathcal O_q^2(\mathbb D)}=S_{p,p}^{-1}$ defined in the previous theorem.
\end{cor}
\subsection{Case of modified Bargmann-Fock spaces $\mathcal B^2(\mathbb C,\nu_{\theta,\beta})$}
Again, to simplify, we fix $-1<\beta_0\leq0$ and $\theta>0$ and for any $n,p\geq0$, we set $\mathfrak{b}_n^p(z)=\sigma_n^p z^{n-p}$ where
$$\sigma_n^p=\sqrt{\frac{\theta^{n-p}(\beta_0+1)_p}{(\beta_0+1)_n}}.$$
Then we know that $(\mathfrak{b}_n^p)_{n\geq0}$ is a Hilbert basis of $\mathcal B^2(\mathbb C,\nu_{\theta,\beta_0+p})$. \\
Analogously to the modified Bergman spaces, we consider here the two subspaces $\mathscr E_p^2(\mathbb C)$ and $\mathscr O_p^2(\mathbb C)$ of $\mathcal B^2(\mathbb C,\nu_{\theta,\beta_0+p})$ generated by the sub-bases $(\mathfrak{b}_{2n}^p)_{n\geq0}$ and  $(\mathfrak{b}_{2n+1}^p)_{n\geq0}$ respectively.\\

For every $p,q\geq 0$, we set
$$\left\{\begin{array}{lcl}
d_{p,q}&=&\ds\sqrt{\frac{(\beta_0+1)_p(\beta_0+1)_q}{\theta^{p+q}}}\\
\ds\mathcal F_{\beta_0}(\xi)&=&\ds d_{p,q}\ _1F_2\left(\left.
\begin{array}{c}
1 \\
\ds\frac{\beta_0+1}{2},\ \ds\frac{\beta_0+2}{2}
\end{array}\right|\left(\frac{\xi}{2}\right)^2\right).
\end{array}\right.$$

\begin{prop}\label{prop3}
For every $p,q\in\mathbb N$ and $\xi\in\mathbb C$, we have
\begin{multicols}{2}
\begin{enumerate}
  \item \quad$\ds\sum_{n=0}^{+\infty} \sigma_n^p\sigma_n^q\xi^n=\ds d_{p,q}\ _1F_1\left(\left.
\begin{array}{c}
1\\
\beta_0+1
\end{array}\right|\theta\xi\right).$
  \item \quad$\ds\sum_{n=0}^{+\infty} \sigma_{2n}^p\sigma_{2n}^q\xi^{2n}=\ds \mathcal F_{\beta_0}(\theta\xi).$
  \item \quad$\ds\sum_{n=0}^{+\infty} \sigma_{2n+1}^p\sigma_{2n+1}^q\xi^{2n}=\ds\frac{\theta }{\beta_0+1}\mathcal F_{\beta_0}(\theta\xi).$
  \item \quad$\ds\sum_{n=0}^{+\infty} \sigma_{2n}^p\sigma_{2n+1}^q\xi^{2n}=\ds\sqrt{\frac{\theta }{\beta_0+1}}\mathcal F_{\beta_0}(\theta\xi).$
\end{enumerate}
\end{multicols}
\end{prop}
Now we can cite some results concerning the Segal-Bargmann transforms between modified Bargmann-Fock spaces and their subspaces. The proofs are similar to the case of modified Bergman spaces (and their subspaces) where we use Proposition \ref{prop3}. We resume these results in a table containing the two spaces $\mathscr H_X,\ \mathscr H_Y$ and their Segal-Bargmann transform $T:\mathscr H_X\longrightarrow \mathscr H_Y$ with  its inverse $T^{-1}:\mathscr H_Y\longrightarrow \mathscr H_X$ given by
$$Tf(z)=\int_Xf(w)\kappa(z,w)d\nu_X(w),\quad T^{-1}g(w)=\int_Yg(z)\tau(w,z)d\nu_Y(z).$$

\begin{tabular}{|c|c|c|c|c|}
  \hline
  $\mathscr H_X$&$\mathscr H_Y$&$\kappa(z,w)$&$\tau(w,z)$\\
  \hline
  $\mathcal B^2(\mathbb C,\nu_{\theta,\beta_0+p})$ &$\mathcal B^2(\mathbb C,\nu_{\theta,\beta_0+q})$ &$\ds \frac{d_{p,q}}{\overline{w}^pz^q}\ _1F_1\left(\left.
\begin{array}{c}
1\\
\beta_0+1
\end{array}\right|\theta\overline{w}z\right)$ &$\ds\frac{d_{p,q}}{\overline{z}^qw^p}\ _1F_1\left(\left.
\begin{array}{c}
1\\
\beta_0+1
\end{array}\right|\theta\overline{z}w\right)$ \\
  \hline
  $\mathscr E_p^2(\mathbb C)$&$\mathscr E_q^2(\mathbb C)$ &$\ds\frac{1}{\overline{w}^pz^q}\mathcal F_{\beta_0}(\theta\overline{w}z)$ & $\ds\frac{1}{\overline{z}^qw^p}\mathcal F_{\beta_0}(\theta\overline{z}w)$ \\
  \hline
  $\mathscr O_p^2(\mathbb C)$&$\mathscr O_p^2(\mathbb C)$ &$\ds\frac{1}{\overline{w}^{p-1}z^{q-1}}\frac{\theta}{\beta_0+1}\mathcal F_{\beta_0}(\theta\overline{w}z)$ & $\ds\frac{1}{\overline{z}^{q-1}w^{p-1}}\frac{\theta}{\beta_0+1}\mathcal F_{\beta_0}(\theta\overline{z}w)$  \\
  \hline
  $\mathscr E_p^2(\mathbb C)$&$\mathscr O_p^2(\mathbb C)$ &$\ds\frac{1}{\overline{w}^pz^{q-1}}\sqrt{\frac{\theta}{\beta_0+1}}\mathcal F_{\beta_0}(\theta\overline{w}z)$ & $\ds\frac{1}{\overline{z}^{q-1}w^p}\sqrt{\frac{\theta}{\beta_0+1}}\mathcal F_{\beta_0}(\theta\overline{z}w)$  \\
  \hline
\end{tabular}

\end{document}